\documentclass{article}
\usepackage{CJK}
\usepackage{comment}
\usepackage{amsmath}
\usepackage{amssymb}
\usepackage{amsthm}
\usepackage{amscd}
\usepackage{graphicx}
\usepackage{indentfirst}
\usepackage{titlesec}
\usepackage[english,francais]{babel}

\newtheorem*{theoa}{Theorem A}

\newtheorem{conj}{Conjecture}

\newtheorem{theo}{Theorem}[section]
\newtheorem{prop}[theo]{Proposition}
\newtheorem{lem}[theo]{Lemma}
\newtheorem{claim}[theo]{Claim}
\newtheorem{cor}[theo]{Corollary}
\newtheorem{rem}[theo]{Remark}
\theoremstyle{definition}
\newtheorem{defi}[theo]{Definition}

\begin{document}
\def \diff {\operatorname{Diff}}
\def \vep {\operatorname{\varepsilon}}
\selectlanguage{english}

\title{On the dominated splitting of Lyapunov stable aperiodic classes}

\author{
Xiaodong Wang
}
\date{}

\maketitle

\begin{abstract}
Recent works related to Palis conjecture of J. Yang, S. Crovisier, M. Sambarino and D. Yang showed that any aperiodic class of a $C^1$-generic diffeomorphism far away from homoclinic bifurcations (or homoclinic tangencies) is partially hyperbolic, see~\cite{csy,y}. We show in this paper that, generically, a non-trivial dominated splitting implies partial hyperbolicity for an aperiodic class if it is Lyapunov stable. More precisely, for $C^1$-generic diffeomorphisms, if a Lyapunov stable aperiodic class has a non-trivial dominated splitting $E\oplus F$, then one of the two bundles is hyperbolic (either $E$ is contracted or $F$ is expanded).
\end{abstract}

\section{Introduction}

One of the goals of dynamical systems is to describe most of (generic or residual, dense) the systems. Examples of Smale~\cite{smale} and others showed that the stable ones (or hyperbolic ones) are not dense in the space of diffeomorphisms, which was thought to be true in the sixties. Palis~\cite{palis,palis2} conjectured that the presence of a homoclinic bifurcation (homoclinic tangency or heterodimensional cycle) is the essential obstacle of hyperbolicity. Related to this conjecture, there are many works. ~\cite{ps} proved this conjecture for the case of dimension two. ~\cite{bgw,c2} proved a weaker version of this conjecture, which states that the union of Morse-Smale systems and the ones with a transverse homoclinic intersection is a dense set of the space of diffeomorphisms (\cite{bgw} solved for dimension three and~\cite{c2} solved for any dimension). In~\cite{c3,cp,csy}, it is proved that far from homoclinic bifurcations (or just far from homoclinic tangencies), the systems have some weak hyperbolicity (partial hyperbolicity or essential hyperbolicity).

Denote by $\diff^1(M)$ the space of $C^1$ self-diffeomorphisms of a smooth compact Riemannian manifold $M$, and assume $f\in\diff^1(M)$. An invariant compact set $K$ is \emph{hyperbolic}, if the tangent bundle can split into two continuous sub-bundles $E^s\oplus E^u$ such that $E^s$ is \emph{contracted} (there are two numbers $m\in\mathbb{N}$ and $0<\lambda<1$, such that, $\|Df^m|_{E^s(x)}\|<\lambda$ holds for all $x\in K$) and $E^u$ is \emph{expanded} (contracted respect to $f^{-1}$). The set $K$ is said to have a \emph{dominated splitting}, if the tangent bundle has a continuous splitting $T_KM=E\oplus F$ and there are two numbers $m\in\mathbb{N}$ and $0<\lambda<1$, such that, for any point $x\in K$, we have $\|Df^m|_{E(x)}\|\cdot\|Df^{-m}|_{F(f^m(x))}\|<\lambda$. To be precise, we also call such a splitting an \emph{$(m,\lambda)$-dominated splitting}. A \emph{partially hyperbolic splitting} over $K$ is a dominated splitting $T_KM=E^s\oplus E^c\oplus E^u$, such that $E^s$ and $E^u$ are contracted respect to $f$ and $f^{-1}$ respectively, and at least one of $E^s$ and $E^u$ is non-trivial. For a periodic point $p$, a \emph{homoclinic tangency} is a non-transverse intersection between the unstable set $W^u(p)$ and stable set $W^s(p)$ of $p$. Two hyperbolic periodic points $p$ and $q$ with different stable dimensions form a \emph{heterodimensional cycle}, if both $W^s(p)\cap W^u(q)$ and $W^s(q)\cap W^u(p)$ are non-empty. It is well known that a diffeomorphism with a homoclinic bifurcation (homoclinic tangency or heteroclinic cycle) is not hyperbolic.

One would like to understand the long behavior of orbits and concentrates on the sets that have some recurrent properties (chain recurrence, non-wandering, recurrent or periodic). By~\cite{conley}, one decomposes the dynamics into pieces which can also be obtained by \emph{pseudo-orbits}, and each piece is called a \emph{chain recurrence class}. For a constant $\vep>0$, a sequence of points $(x_n)_{n=a}^b$ is called a \emph{$\vep$-pseudo-orbit}, if for any $a\leq n<b$, we have $d(f(x_n),x_{n+1})<\vep$, where $-\infty\leq a<b\leq +\infty$. A point $y$ is called \emph{chain attainable} from $x$, denoted by $x\dashv y$, if for any $\vep>0$, there is a $\vep$-pseudo-orbit $\{x=x_0,x_1,\cdots,x_n=y\}$. The \emph{chain recurrent set} $R(f)$ is the invariant compact set of points $x$ such that $x\dashv x$, where such a point $x$ is called a \emph{chain recurrent point}. On $R(f)$, one can define an equivalent relation $x\sim y$, if and only if $x\dashv y$ and $y\dashv x$. The \emph{chain recurrence class} of $f$ is the equivalent class of $\sim$ on $R(f)$. Obviously, $R(f)$ contains all periodic points of $f$. A chain-recurrent class that contains no periodic point is called an \emph{aperiodic class}.

Recall that a subset $R$ of a topological Baire space $X$ is called a \emph{residual} set, if it contains a dense $G_{\delta}$ set of $X$. A property is a \emph{generic} property of $X$, if there is a residual set $R\subset X$, such that each element contained in $R$ satisfies the property. An invariant compact set $K$ is called \emph{Lyapunov stable}, if for any neighborhood $U$ of $K$, there is another neighborhood $V$ of $K$, such that $f^n(V)\subset U$ for all $n\geq 0$.

In~\cite{bc}, Bonatti and Crovisier proved that a chain recurrence class of a $C^1$ generic diffeomorphism is either a \emph{homoclinic class} (see Definition~\ref{homoclinic class}) or an aperiodic class. In~\cite{c3,csy,y}, it is proved that, any aperiodic class of a $C^1$-generic diffeomorphism that is far from homoclinic bifurcations (or just homoclinic tangencies) is partially hyperbolic with center bundle of dimension one. ~\cite{cp,y1} proves that any aperiodic class of a $C^1$-generic diffeomorphism that is far from homoclinic bifurcations can not be Lyapunov stable. But in~\cite{bd}, it is showed that there is an open set $\mathcal{U}\subset \diff^1(M)$, such that for generic $f\in\mathcal{U}$, there are infinitely many aperiodic classes that are Lyapunov stable both respect to $f$ and $f^{-1}$. In~\cite{po}, Potrie proved that for $C^1$-generic diffeomorphisms, if a homoclinic class is Lyapunov stable both for $f$ and for $f^{-1}$, then it admits a non-trivial dominated splitting, and under some more hypothesis, it is the whole manifold. There are also many other results for Lyapunov stable homoclinic classes in~\cite{po,acco}. In~\cite{c4}, Crovisier raised a conjecture for aperiodic classes, which implies that, $C^1$-generically, if an aperiodic class admits a dominated splitting, then one of the bundle is hyperbolic.

\begin{conj}[\cite{c4}]\label{conjecture of crovisier}
Let $f\in\diff^1(M)$ be a $C^1$-generic diffeomorphism and $\Lambda$ be an aperiodic class of $f$. Assume that $E^s\oplus E^c\oplus E^u$ is the dominated splitting on $\Lambda$ such that $E^s$ (resp. $E^u$) is the maximal contracted (resp. expanded) sub-bundle, then $E^c$ has dimension at least two and admits no non-trivial dominated splitting.
\end{conj}

Now we state our main theorem, which claims that a non-trivial dominated splitting on a Lyapunov stable aperiodic class is actually a partially hyperbolic splitting for $C^1$-generic diffeomorphisms. This gives a partial answer to Conjecture~\ref{conjecture of crovisier}.

\begin{theoa}
There is a residual subset $\mathcal{R}\subset \diff^1(M)$, such that, for any $f\in\mathcal{R}$, if a Lyapunov stable aperiodic class $\Lambda$ of $f$ admits a dominated splitting $T_{\Lambda}M=E\oplus F$, then either $E$ is contracted or $F$ is expanded.
\end{theoa}

As a consequence of Theorem A, for $C^1$-generic diffeomorphisms, if we consider the dominated splitting $E^s\oplus E^c\oplus E^u$ on a Lyapunov stable aperiodic class, such that $E^s$ (resp. $E^u$) is the maximal contracted (resp. expanded) sub-bundle, then the sub-bundle $E^c$ admits no non-trivial dominated splitting. Moreover, with the arguments of~\cite{c3,cp}, one knows that the dimension of $E^c$ is at least two. Hence Conjecture~\ref{conjecture of crovisier} holds for Lyapunov stable aperiodic classes.

We point out that the statement of Theorem A fails for homoclinic classes, which shows that the aperiodicity is an essential assumption. ~\cite{bv} constructs a robustly transitive diffeomorphism of $\mathbb{T}^4$, hence the whole manifold is a bi-Lyapunov stable homoclinic class, and it admits only one dominated splitting $E\oplus F$ with $dim(E)=2$. Moreover, there are periodic saddles of all possible stable dimensions, which implies that neither $E$ is contracted nor $F$ is expanded.

In~\cite{bgy,w}, it is proved for homoclinic classes (see Definition~\ref{homoclinic class}) that one bundle is hyperbolic under some assumptions for the other bundle and for the periodic orbits. By Theorem A, we can get that one (only one actually) of the two bundles $E$ and $F$ is hyperbolic, but we do not know which one it is. For any point $x$ contained in a Lyapunov stable chain recurrence class, the unstable set of $x$, $W^u(x)=\{y\in M: \lim_{n\rightarrow+\infty}d(f^{-n}(x),f^{-n}(y))=0\}$ is also contained in the class. Hence if the bundle $F$ is expanded, by~\cite{hps}, the Lyapunov stable aperiodic class is foliated by unstable manifolds that are tangent to $F$, and thus it can not be minimal. We conjecture that such phenomenon can not happen.

\begin{conj}\label{hyperbolic bundle}
For $C^1$-generic $f\in\diff^1(M)$, if a Lyapunov stable aperiodic class $\Lambda$ of $f$ admits a dominated splitting $T_{\Lambda}M=E\oplus F$, then the bundle $E$ is contracted.
\end{conj}

Bonatti and Shinohara have a programme to construct (Lyapunov stable) aperiodic classes with a non-trivial dominated splitting. Since aperiodic classes are not isolated, such examples are not easy to construct, even for non-isolated homoclinic classes. In a recent paper, they build non-isolated Lyapunov stable homoclinic classes with a non-trivial dominated splitting on any 3-manifold, see~\cite{bs2}. The main tool is the notion of flexible periodic points introduced in~\cite{bs1}.

Actually, each of the bi-Lyapunov stable aperiodic classes constructed in~\cite{bd} is a minimal Cantor set, and admits no non-trivial dominated splitting. Hence we have the second conjecture that is obviously true if Conjecture~\ref{hyperbolic bundle} is true.

\begin{conj}\label{no domination for bi-stable class}
For $C^1$-generic $f\in\diff^1(M)$, if an aperiodic class $\Lambda$ of $f$ is Lyapunov stable both respect to $f$ and $f^{-1}$, then it admits no non-trivial dominated splitting.
\end{conj}

There is a result related to Conjecture~\ref{no domination for bi-stable class}, see~\cite{zh}. It is proved that, if a diffeomorphism is minimal, hence the whole manifold is a bi-Lyapunov stable aperiodic class, then it admits no dominated splitting.

To prove Theorem A, we have to show that under the hypothesis, any chain recurrence class contains a periodic point if neither of the two bundles in the domination is hyperbolic. We have to use the following proposition to obtain periodic orbits that spend most of the time close to an invariant compact set and visit a small neighborhood of a point. Recall that an invariant compact set $K$ is called a \emph{chain transitive set}, if for any $\vep>0$, there is a periodic $\vep$-pseudo-orbit contained in $K$ and $\vep$-dense in $K$.

\begin{prop}\label{asymptotic connecting}
For $C^1$-generic $f\in\diff^1(M)$, assume $K$ is a chain transitive set of $f$ and $x\in K$. If $x\notin \alpha(x)$, then for any $C^1$-neighborhood $\mathcal{U}$ of $f$, any neighborhood $U$ of $\alpha(x)$, any neighborhood $U_x$ of $x$ and any neighborhood $U_K$ of $K$, there is an integer $L\in\mathbb{N}$, with the following property. For any integer $m\in\mathbb{N}$, there is a diffeomorphism $g\in\mathcal{U}$ with a periodic point $p\in U_x$ whose orbit is contained in $U_K$, satisfying that:
\begin{itemize}
\item $\#(orb(p,g)\cap U)\geq m$,
\item $\#(orb(p,g)\setminus U)\leq L$.
\end{itemize}
\end{prop}

\begin{rem}\label{rem of asymptotic connecting}
$(1)$ Clearly, if we replace $\alpha(x)$ by $\omega(x)$ in the hypothesis, the conclusions are still valid. We point out here that, in the proof of Theorem A, we use the assumption that $x\notin \omega(x)$. But to simplify the notations, we prove Proposition~\ref{asymptotic connecting} under the assumption $x\notin \alpha(x)$.

$(2)$ For the proof of Theorem A, we only have to consider the case where $K$ contains no periodic point. But to give a general statement of Proposition~\ref{asymptotic connecting}, we will also prove the case when $K$ contains periodic points.
\end{rem}

\section*{Acknowledgements}

The author is grateful to Sylvain Crovisier and Dawei Yang for carefully listening to the proof and for the useful discussions. Sylvain Crovisier gave me lots of guidance in the writing and Dawei Yang proposed Conjecture~\ref{no domination for bi-stable class} to me. The author would like to thank Lan Wen, Shaobo Gan, Rafael Potrie and Jinhua Zhang for carefully reading the proof and for the useful comments. The author would also like to thank University Paris-Sud 11 for the hospitality and China Scholarship Council (CSC) for financial support (201306010008).

\section{Preliminary}

We give some definitions and known results in this section.

\subsection{Decomposition of dynamics}

Periodic points have the best recurrent property. One can decompose the closure of hyperbolic periodic points by a relation called \emph{homoclinic relation}.

\begin{defi}\label{homoclinic class}
Assume $p$ and $q$ are two hyperbolic periodic points of a diffeomorphism $f$. They are called \emph{homoclinically related}, if $W^u(orb(p))$ and $W^s(orb(p))$ have non empty transverse intersections with $W^s(orb(q))$ and $W^u(orb(q))$ respectively. The \emph{homoclinic class} of $p$, denoted by $H(p,f)$ (or $H(p)$ if there is no confusion), is the closure of the set of periodic points that are homoclinically related to $p$.
\end{defi}

We give a relation $\prec$ that is first introduced in~\cite{a} and~\cite{gw}, also see~\cite{c1}.

\begin{defi}
Consider a diffeomorphism $f\in\diff^1(M)$. For any two points $x,y\in M$, we denote $x\prec y$ if for any neighborhood $U$ of $x$ and any neighborhood $V$ of $y$, there are a point $z\in M$ and an integer $n\geq 1$, such that $z\in U$ and $f^n(z)\in V$.
\end{defi}

Sometimes one has to localize the dynamics for the two relations $\dashv$ and $\prec$, see~\cite{c1}.

\begin{defi}
Consider a diffeomorphism $f\in\diff^1(M)$. Assume that $K$ is a compact set of $M$ and $W$ is an open set of $M$. We denote $x\prec_W y$ if for any neighborhood $U$ of $x$ and any neighborhood $V$ of $y$, there is a piece of orbit $(z,f(z),\cdots,f^n(z))$ contained in $W$ such that $z\in U$ and $f^n(z)\in V$. We denote $x\prec_K y$, if for any neighborhood $U$ of $K$, one has $x\prec_U y$. For a point $x$ and a compact set $\Lambda$, we denote $x\prec \Lambda$ (resp. $x\prec_{W} \Lambda$ and $x\prec_{K} \Lambda$), if for any point $y\in \Lambda$, we have $x\prec y$ (resp. $x\prec_{W} y$ and $x\prec_{K} y$). Similarly, we define $x\dashv_{W} y$ and $x\dashv_{K} y$ (resp. $x\dashv_{W} \Lambda$ and $x\dashv_{K} \Lambda$).
\end{defi}

For the relation $\prec$, we have the following result, see Lemma 6 in~\cite{c1}.

\begin{lem}\label{prec}
Assume that $K$ is an invariant compact set, then for any neighborhoods $U_2\subset U_1$ of $K$ and any point $y\in U_1$ satisfying $y\prec_{U_1} K$, there is a point $y'\in U_2$, such that $y\prec_{U_1} y'\prec_{U_2} K$ and the positive orbit of $y'$ is contained in $U_2$.
\end{lem}

One has the $C^1$ connecting lemma to connect two orbits by perturbation, see~\cite{h,wx}.

\begin{theo}\label{connecting lemma}
Assume $f$ is a diffeomorphism in $\diff^1(M)$. For any neighborhood $\mathcal{U}$ of $f$, there is an integer $N\in\mathbb{N}$, satisfying the following property:\\
For any point $x$ that is not a periodic point of $f$ with period less than or equal to $N$, for any neighborhood $V_x$ of $x$ there is a neighborhood $V'_x\subset V_x$, such that, for any two points $p,q\notin \bigcup_{i=0}^{N-1}f^i(V_x)$, if $p$ has a positive iterate $f^{n_p}(p)\in V'_x$ and $q$ has a negative iterate $f^{-n_q}(q)\in V'_x$, where $n_p,n_q\in \mathbb{N}$, then there is a diffeomorphism $g\in \mathcal{U}$ that coincides with $f$ outside $\bigcup_{i=0}^{N-1}f^i(V_x)$ and $q$ is on the positive orbit of $p$. Moreover, assume $g^m(p)=q$, then $m\leq n_p+N+n_q$ and $\{p,g(p),\cdots,g^m(p)=q\}\subset \bigcup_{i=0}^{n_p}\{f^i(p)\}\cup \bigcup_{i=0}^{N-1}f^i(V_x)\cup \bigcup_{i=0}^{n_q}\{f^{-i}(q)\}$.
\end{theo}

\subsection{Pliss points and weak sets}

\begin{defi}\label{pliss point}
Assume there is a dominated splitting $T_KM=E\oplus F$ over an invariant compact set $K$ of a diffeomorphism $f\in\diff^1(M)$ and $0<\lambda<1$. A point $x\in K$ is called a \emph{$\lambda$-$E$-Pliss point} (resp. \emph{$\lambda$-$F$-Pliss point}), if $\prod_{i=0}^{n-1} \|Df|_{E(f^{i}(x))}\|\leq {\lambda}^n$ (resp. $\prod_{i=0}^{n-1} \|Df^{-1}|_{F(f^{-i}(x))}\|\leq {\lambda}^n$) holds for any $n\geq 1$. If $x$ is both a $\lambda$-$E$-Pliss point and a $\lambda$-$F$-Pliss point, then it is called a \emph{$\lambda$-bi-Pliss point}. Two $\lambda$-$E$-Pliss points ($f^k(x)$, $f^l(x)$) on an orbit $orb(x)$ are called \emph{consecutive $\lambda$-$E$-Pliss points}, if $k<l$ and $f^i(x)$ is not a $\lambda$-$E$-Pliss point for any $k<i<l$. Similarly we define \emph{consecutive $\lambda$-$F$-Pliss points}.
\end{defi}

\begin{rem}\label{uniform manifold for pliss points}
It is well known that the stable manifold of a $\lambda$-$E$-Pliss point has a uniform scale of dimension $dim(E)$ which depends only on the diffeomorphism $f$.
\end{rem}

\begin{defi}\label{weak set}
Consider a diffeomorphism $f\in\diff^1(M)$ and a constant $0<\lambda<1$. An invariant compact set $K$ with a dominated splitting $T_KM=E\oplus F$ is called a \emph{$\lambda$-$E$-weak set} (resp. $\lambda$-$F$-weak set), if there is no $\lambda$-$E$-Pliss point (resp. $\lambda$-$F$-Pliss point) contained in $K$.
\end{defi}

\begin{rem}\label{F-weak implies E-uniform}
Assume that an invariant compact set $K$ is a $\lambda'$-$F$-weak set where $T_KM=E\oplus F$ is a $(1,\lambda^2)$-dominated splitting with $0<\lambda<\lambda'<1$. By Definitions~\ref{pliss point} and~\ref{weak set}, for any point $y$ contained in $K$, there is an integer $n_y\geq 1$, such that $\prod_{i=0}^{n_y-1}\|Df^{-1}|_{F(f^{-i}(y))}\|> (\lambda')^{n_y}$. Then one has that $\prod_{i=0}^{n_y-1}\|Df|_{E(f^{i}(f^{-n_y}(y)))}\|< \left(\frac{\lambda^2}{\lambda'}\right)^{n_y}<\lambda^{n_y}$. By the compactness of the set $K$, the integers $n_y$ are uniformly bounded. Hence $E|_K$ is uniformly contracted.
\end{rem}

For the properties of bi-Pliss points, we state a lemma here, whose proof will be omitted. The arguments can be seen in papers like~\cite{ps}.

\begin{lem}\label{property of pliss point}
Consider a diffeomorphism $f\in \diff^1(M)$. Assume that there is a $(1,\lambda^2)$-dominated splitting $T_KM=E\oplus F$ over an invariant compact set $K$ for some $0<\lambda<1$. For any number $\lambda'\in(\lambda,1)$, the following properties are satisfied:
\begin{enumerate}
\item Assume there is a sequence of consecutive $\lambda'$-$E$-Pliss points $(f^{k_n}(x_n),f^{l_n}(x_n))_{n\geq 1}$ such that $l_n-k_n\rightarrow +\infty$ as $n\rightarrow\infty$, and $y$ is a limit point of the sequence $(f^{l_n}(x_n))_{n\geq 1}$, then $y$ is a $\lambda'$-bi-Pliss point.

\item For any $x\in K$, if there are $\lambda'$-$E$-Pliss points on $orb^+(x)$ and $\lambda'$-$F$-Pliss points on $orb^-(x)$, then there is at least one $\lambda'$-bi-Pliss point on $orb(x)$.

\item For any $x\in K$, if $x$ is a $\lambda'$-$E$-Pliss point and there is no $\lambda'$-$E$-Pliss point on $orb^-(x)\setminus \{x\}$, then $x$ is a $\lambda'$-bi-Pliss point.
\end{enumerate}
\end{lem}

To obtain Pliss points, one can use the following lemma given by V. Pliss, see~\cite{p},~\cite{ps}.

\begin{lem}[Pliss Lemma]\label{pliss lemma}
Consider a diffeomorphism $f\in\diff^1(M)$ and two numbers $0<\lambda_1<\lambda_2<1$. Assume $K$ is an invariant compact set and $E\subset T_KM$ is $Df$-invariant. Then the following properties are satisfied:
\begin{enumerate}
\item There are an integer $N$ and a number $c$ that depend only on $f$, $\lambda_1$ and $\lambda_2$, such that, for any point $x\in K$, and any number $n\geq N$, if
   \begin{displaymath}
     \prod_{i=0}^{n-1}\|Df|_{E(f^ix)}\|\leq {\lambda_1}^n,
   \end{displaymath}
then, there are $r$ integers $0\leq n_1<\cdots<n_r\leq n$, where $r>cn$, such that, for any $j=1,\cdots,r$ and any $k=n_j+1,\cdots,n$, we have:
   \begin{displaymath}
     \prod_{i=n_j}^{k-1}\|Df|_{E(f^ix)}\|\leq {\lambda_2}^{k-n_j}.
   \end{displaymath}

\item For any $x\in K$, any integer $l$, if
   \begin{displaymath}
     \prod_{i=0}^{n-1}\|Df|_{E(f^ix)}\|\leq {\lambda_1}^n,
   \end{displaymath}
holds for any $n\geq l$, then there are a sequence of integers $0\leq n_1<n_2\cdots<n_r<\cdots$, such that for any $j\geq 1$ and any $k\geq n_j+1$, we have:
   \begin{displaymath}
     \prod_{i=n_j}^{k-1}\|Df|_{E(f^ix)}\|\leq {\lambda_2}^{k-n_j}.
   \end{displaymath}
\end{enumerate}
\end{lem}

We have the following corollary of the Pliss lemma. The proof can be found in~\cite{w}.

\begin{cor}\label{corollary of pliss}
Assume there is a dominated splitting $T_KM=E\oplus F$ on an invariant compact set $K$ of a diffeomorphism $f$. Then for any $x\in K$, and any number $0<\lambda<1$, the following two properties are satisfied.
\begin{enumerate}
\item If $x$ is a $\lambda$-$E$-Pliss point, then $\omega(x)$ contains some $\lambda$-$E$-Pliss points.

\item If for any point $y\in\omega(x)$, there is an integer $n_y\geq 1$, such that
   \begin{displaymath}
     \prod_{i=0}^{n_y-1}\|Df|_{E(f^iy)}\|\leq {\lambda}^{n_y},
   \end{displaymath}
then for any $\lambda'\in(\lambda,1)$, there are infinitely many $\lambda'$-$E$-Pliss point on $orb^+(x)$.
\end{enumerate}
\end{cor}

We have the selecting lemma of Liao to obtain weak periodic orbits, see~\cite{l,w-selecting}.

\begin{lem}[Liao's selecting lemma]\label{selecting}
Consider an invariant compact set $\Lambda$ of a diffeomorphism $f\in\diff^1(M)$ with a non-trivial $(1,\lambda^2)$-dominated splitting $T_\Lambda M=E\oplus F$. Assume that the following two conditions are satisfied:
\begin{itemize}
\item There is a point $b\in \Lambda$, such that, for all $n\geq 1$, we have:

   \begin{displaymath}
     \prod_{i=0}^{n-1} \|Df|_{E(f^{i}(b))}\|\geq 1.
   \end{displaymath}

\item There is a constant $\lambda_0\in(\lambda,1)$, such that there is no $\lambda_0$-$E$-weak set contained in $\Lambda$.
\end{itemize}
Then for any two numbers $\lambda_1,\lambda_2\in(\lambda_0,1)$ with $\lambda_1<\lambda_2$, there is a sequence of periodic orbits $orb(q_n)$ that are homoclinically related with each other and converges to a subset of $\Lambda$, such that, denoting by $\tau(q_n)$ the period of $p_n$, then for any $n=0,1,2,\cdots,\tau(q_n)$,

   \begin{displaymath}
     \prod_{i=0}^{n-1} \|Df|_{E(f^{i}(q_n))}\|\leq {\lambda_2}^n,
   \end{displaymath}
and
   \begin{displaymath}
     \prod_{i=n}^{\tau(q_n)-1} \|Df|_{E(f^{i}(q_n))}\|\geq {\lambda_1}^{\tau(q_n)-n}.
   \end{displaymath}
Similar assertions for $F$ hold with respect to $f^{-1}$.
\end{lem}

\subsection{Generic properties}

We give some known $C^1$-generic properties of diffeomorphisms in the following theorem. These results can be found in books or papers like~\cite{bdv,bc,c1} etc.

\begin{theo}\label{generic properties}
There is a residual set $\mathcal{R}_0$ in $\diff^1(M)$, such that any diffeomorphism $f\in\mathcal{R}_0$ satisfies the following properties:
\begin{enumerate}
\item The diffeomorphism $f$ is Kupka-Smale: all periodic points of $f$ are hyperbolic and the stable and unstable manifolds of periodic orbits intersect transversely.

\item The periodic points are dense in the chain recurrent set and any chain recurrence class containing a periodic point $p$ is the homoclinic class of $p$.

\item For any two points $x,y\in M$ and a compact set $K$, $x\dashv_K y$ if and only if $x\prec_K y$.

\item Any chain transitive compact invariant set can be accumulated by a sequence of hyperbolic periodic orbits in the Hausdorff distance.
\end{enumerate}
\end{theo}

\section{No hyperbolic bundle implies existence of periodic points: proof of Theorem A}

In this section, we give the proof of Theorem A.

First, we take $f\in\mathcal{R}_0$ satisfying the assumptions of Theorem A, where $\mathcal{R}_0$ is the residual subset in $\diff^1(M)$ that satisfies the conclusions of Theorem~\ref{generic properties}. Also, we assume that $f$ satisfies the properties stated in Proposition~\ref{asymptotic connecting}. Later we will assume also that $f$ belongs to another residual subset $\mathcal{R}_1$ of $\diff^1(M)$, which will be defined below.

By taking an adapted metric~\cite{g}, we assume that $T_{\Lambda}M=E\oplus F$ is a $(1,\lambda^2)$-dominated splitting for a constant $0<\lambda<1$. To simplify the notations, we call $E\oplus F$ a $\lambda^2$-dominated splitting. It is known that a dominated splitting of an invariant compact set of $f$ can be extended to the maximal invariant compact set of a small neighborhood for any diffeomorphism $g$ that is $C^1$-close to $f$. Then we can take a neighborhood $V_0$ of $\Lambda$ and a $C^1$-neighborhood $\mathcal{U}_0$ of $f$, such that, for any $g\in\mathcal{U}_0$, the invariant compact set $\bigcap_{n\in\mathbb{Z}}g^n(\overline{V_0})$ has a $\lambda^2$-dominated splitting that is the extension of $T_{\Lambda}M=E\oplus F$. To simplify the notations, we still denote this dominated splitting by $E\oplus F$.

We assume by absurd that neither $E|_{\Lambda}$ is contracted nor $F|_{\Lambda}$ is expanded. Take two numbers $\lambda_1<\lambda_2$ in the interval $(\lambda,1)$. In order to find a contradiction, one has to obtain a periodic orbit which intersects the aperiodic class $\Lambda$.

\subsection{Existence of a bi-Pliss point whose $\omega$-limit set is $E$ contracted}

Since $F$ is not expanded, there is a point $b\in\Lambda$, such that $\prod_{i=0}^{n-1} \|Df^{-1}|_{F(f^{-i}(b))}\|\geq 1$ for any $n\geq 1$. By Lemma~\ref{selecting}, if there is no $\lambda_1$-$F$-weak set contained in $\Lambda$, then $\Lambda$ intersects a homoclinic class, and hence $\Lambda$ is contained in this homoclinic class by Theorem~\ref{generic properties}, which is a contradiction. Thus $\Lambda$ contains $\lambda_1$-$F$-weak sets. Moreover, by Remark~\ref{F-weak implies E-uniform}, the whole aperiodic class $\Lambda$ is not a $\lambda_1$-$F$-weak set, since we have assumed that $E|_{\Lambda}$ is not contracted. Following the same arguments of Section 4.1 of~\cite{w}, we can get the following lemma.

\begin{lem}\label{pliss point x and omega x}
There is a $\lambda_2$-bi-Pliss point $x\in\Lambda$, such that $\omega(x)$ is a $\lambda_1$-$F$-weak set and $x\notin\omega(x)$.
\end{lem}

The statement in Section 4.1 of~\cite{w} is for homoclinic classes and for the assumption that the bundle $E$ is not contracted. But the proof of Lemma~\ref{pliss point x and omega x} follows exactly the same arguments, so we only give a short explanation here.

\begin{proof}[Sketch of proof] We take the closure of the union of all $\lambda_1$-$F$-weak sets contained in $\Lambda$ and denote it by $\hat{K}$. We consider two cases: either $\hat{K}$ is still a $\lambda_1$-$F$-weak set or not.

If $\hat{K}$ is a $\lambda_1$-$F$-weak set, then one can prove that it is locally maximal in the aperiodic class $\Lambda$ (in this case, $\hat K\subsetneqq\Lambda$). Hence there is a point $z\in \Lambda\setminus \hat{K}$, such that $\omega(z)\subset \hat{K}$. By the maximality of the $\lambda_1$-$F$-weak property of $\hat{K}$, one can prove more that, there is at least one $\lambda_1$-$F$-Pliss point contained in $\alpha(z)$. By the domination of $E\oplus F$, for any $y\in\omega(z)$, there is an integer $n_y\geq 1$, such that $\prod_{i=0}^{n_y-1}\|Df|_{E(f^i{y})}\|<\lambda^{n_y}$. Hence there are infinitely many $\lambda_1$-$E$-Pliss points on $orb^+(z)$, by item $(2)$ of Corollary~\ref{corollary of pliss}. We then consider whether there are finitely or infinitely many $\lambda_1$-$E$-Pliss points on $orb^-(z)$. Using Lemma~\ref{property of pliss point} and Corollary~\ref{corollary of pliss}, one can get the conclusion in both subcases.

If otherwise $\hat{K}$ is not a $\lambda_1$-$F$-weak set, then one can prove that for any integer $T$, there are a $\lambda_1$-$F$-weak set $K\subset \Lambda$ and a $\lambda_1$-$E$-Pliss point $z\in K$, such that, for any $n\leq T$,
   \begin{displaymath}
     \prod_{i=0}^{n-1}\|Df^{-1}|_{F(f^{-1}(z))}\|\leq\lambda_1^n.
   \end{displaymath}
By choosing a sequence of integers $T_n\rightarrow+\infty$, one can get a sequence of points $z_n\in \Lambda$, such that $z_n$ converges to a $\lambda_1$-bi-Pliss point $y\in\Lambda\setminus K$. By Remark~\ref{uniform manifold for pliss points}, one can see that $W^s(z_n)\cap W^u(y)\neq\emptyset$ when $n$ is large enough. Then the intersection point $x\in W^s(z_n)\cap W^u(y)$ is a $\lambda_2$-bi-Pliss point if the points $z_n$ and $y$ are close enough. Moreover, one can see that $\omega(x)\subset K$ is a $\lambda_1$-$F$-weak set and $x\in\Lambda\setminus K$.
\end{proof}

\subsection{Existence of $E$-contracted periodic orbits by perturbations}

We take the $\lambda_2$-bi-Pliss point $x\in\Lambda$ from Lemma~\ref{pliss point x and omega x}. Then we have that $\omega(x)$ is a $\lambda_1$-$F$-weak set and $x\notin\omega(x)$. We have the following lemma to get $E$-uniformly contracted periodic orbits close to $\Lambda$ by $C^1$-small perturbations.

\begin{lem}\label{perturbations for E contracted orbits}
For any $C^1$-neighborhood $\mathcal{U}$ of $f$, any neighborhood $V$ of $\Lambda$, and any neighborhood $U_x$ of $x$, there are a diffeomorphism $g\in\mathcal{U}$ and a periodic point $q\in U_x$ of $g$ with period $\tau$, such that $orb(q,g)\subset V$, and
   \begin{displaymath}
     \prod_{i=0}^{\tau-1}\|Dg|_{E(g^i(q))}\|< {\lambda_1}^{\tau}.
   \end{displaymath}
\end{lem}

\begin{proof}
Take a $C^1$-neighborhood $\mathcal{U}$ of $f$, a neighborhood $V$ of $\Lambda$ and a neighborhood $U_x$ of $x$. Without loss of generality, we assume that $\mathcal{U}\subset\mathcal{U}_0$, $V\subset V_0$ and $U_x\subset V_0$.

By Remark~\ref{F-weak implies E-uniform} and the compactness of $\omega(x)$, there is an integer $T\geq 1$, such that, for any $y\in\omega(x)$,
   \begin{displaymath}
     \prod_{i=0}^{T-1}\|Df|_{E(f^iy)}\|< {\lambda}^{T}.
   \end{displaymath}
Then there are a neighborhood $U\subset V$ of $\omega(x)$, and a $C^1$-neighborhood $\mathcal{V}\subset\mathcal{U}$ of $f$, such that, for any point $y$ with $orb(y,g)\subset V$, if $g^i(y)\subset U$ for any $0\leq i\leq T$, then it holds that
   \begin{displaymath}
     \prod_{i=0}^{T-1}\|Dg|_{E(g^iy)}\|< {\lambda}^{T}.
   \end{displaymath}

Take $C=sup\{\|Dg\|:g\in\mathcal{V}\}$. Take a small neighborhood $U_x\subset V$ of $x$. By Proposition~\ref{asymptotic connecting}, there is an integer $L$ associated to $(f,\mathcal{V},U,U_x,V)$. Take an integer $k\in\mathbb{N}$, such that, for any $n\geq k$, $\lambda^{nT}\cdot C^{LT+T}<\lambda_1^{nT+LT+T}$. Take $m=kT+LT+T$, then by Proposition~\ref{asymptotic connecting}, there are a diffeomorphism $g\in\mathcal{V}\subset\mathcal{U}$ and a periodic point $q\in U_x$ of $g$, such that $orb(q,g)\subset V$ satisfies $\#(orb(p,g)\cap U)\geq m$, and $\#(orb(p,g)\setminus U)\leq L$. Denote by $\tau$ the period of $q$ under the iterate of $g$. It can be written as $\tau=nT+LT+r$, where $0\leq r< T$. Then by the distribution of the points of $orb(q,g)$, there are at least $n$ pieces of segments $\{f^{l_i}(q),f^{l_i+1}(q),\cdots,f^{l_i+T-1}(q)\}_{1\leq i\leq n}$ that are pairwise disjoint and contained in $U$. Hence we have that
   \begin{displaymath}
     \prod_{i=0}^{\tau-1}\|Dg|_{E(g^i(q))}\|\leq C^{LT+T}\cdot \prod_{j=1}^{n}\prod_{i=l_j}^{l_j+T-1}\|Dg|_{E(g^i(q))}\|\leq {\lambda}^{nT}\cdot N^{LT+T}<{\lambda_1}^{nT+LT+T}< {\lambda_1}^{\tau}.
   \end{displaymath}
This finishes the proof of Lemma~\ref{perturbations for E contracted orbits}.
\end{proof}

\subsection{Existence of $E$ contracted periodic orbits for $f$: a Baire argument}

From Lemma~\ref{perturbations for E contracted orbits}, we obtain some $E$ contracted periodic orbits close to the aperiodic class $\Lambda$ that has a point close to the $\lambda_2$-bi-Pliss point $x\in\Lambda$ by $C^1$-small perturbations of $f$. Then, with a standard Baire argument (see for example~\cite{gw}), we can obtain such periodic orbits for the generic diffeomorphism $f$ itself.

\begin{lem}\label{E contracted orbits for f}
There is a residual subset $\mathcal{R}_1\subset\diff^1(M)$, such that, if $f\in\mathcal{R}_1$, then for any neighborhood $V$ of $\Lambda$, and any neighborhood $U_x$ of the $\lambda_2$-bi-Pliss point $x$, there is a periodic point $q\in U_x$ of $f$ with period $\tau$, such that $orb(q)\subset V$, and
   \begin{displaymath}
     \prod_{i=0}^{\tau-1}\|Df|_{E(f^i(q))}\|\leq {\lambda_1}^{\tau}.
   \end{displaymath}
\end{lem}

\begin{proof}
Take a countable basis $(U_m)_{m\geq 1}$ of $M$. Denote by $(V_n)_{n\geq 1}$ the countably collection of sets such that each $V_n$ is a union of finitely many sets of the basis $(U_m)_{m\geq 1}$.

Let $\mathcal{H}_{m,n,j}$ be the set of $C^1$ diffeomorphisms $h$, satisfying the following properties. \\
\emph{There is a hyperbolic periodic orbit $orb(q,h')$, such that
\begin{itemize}
\item $orb(q,h')\subset V_n$ and $orb(q,h')\cap U_m\neq \emptyset$,
\item there is a dominated splitting $T_{orb(q,h')}M=E\oplus F$ with $dim(E)=j$, and, denoting by $\tau$ the period of $q$, then
   \begin{displaymath}
     \prod_{i=0}^{\tau-1}\|Dh'|_{E(h'^i(q))}\|< {\lambda_1}^{\tau}.
   \end{displaymath}
\end{itemize}}

Notice that $\mathcal{H}_{m,n,j}$ is an open subset of $\diff^1(M)$. Take $\mathcal{N}_{m,n,j}=\diff^1(M)\setminus\overline{\mathcal{U}_{m,n,j}}$, then the set $\mathcal{H}_{m,n,j}\cup\mathcal{N}_{m,n,j}$ is an open and dense subset of $\diff^1(M)$.

Let
   \begin{center}
     $\mathcal{R}_1=\mathcal{R}_0\cap (\bigcap_{m,n\geq 1,1\leq j\leq d-1}(\mathcal{H}_{m,n,j}\cup\mathcal{N}_{m,n,j}))$,
   \end{center}
where $\mathcal{R}_0$ is taken from Theorem~\ref{generic properties}. Then the set $\mathcal{R}_1$ is a residual subset of $\diff^1(M)$. We now prove that the conclusion of Lemma~\ref{E contracted orbits for f} is valid for the residual subset $\mathcal{R}_1$.

Take any diffeomorphism $f\in\mathcal{R}_1$. For any neighborhood $V$ of $\Lambda$, and any neighborhood $U_x$ of $x$, there are two integers $m$ and $n$, such that $V_n\subset V$ and $x\in U_m\subset U_x$. By Lemma~\ref{perturbations for E contracted orbits}, for any neighborhood $\mathcal{V}$ of $f$, there is a diffeomorphism $g\in\mathcal{V}$ such that $g\in \mathcal{H}_{m,n,j}$, where $j=dim(E)$. This means that $f\in\overline{\mathcal{H}_{m,n,j}}$. Then $f\notin \mathcal{N}_{m,n,j}$, and thus $f\in \mathcal{H}_{m,n,j}$. The proof of Lemma~\ref{E contracted orbits for f} is finished by the construction of $\mathcal{H}_{m,n,j}$.
\end{proof}

\subsection{The aperiodic class $\Lambda$ hits periodic orbits: a contradiction}

Let $\mathcal{R}=\{f\in\diff^1(M):f\in\mathcal{R}_0\cap\mathcal{R}_1,\text{ and $f$ satisfies the properties in Proposition~\ref{asymptotic connecting}}\}$. Then $\mathcal{R}$ is a residual subset of $\diff^1(M)$. We have the following lemma.

\begin{lem}\label{contradiction}
For any diffeomorphism $f\in\mathcal{R}$ that satisfies the assumptions above, there is a point $y\in \Lambda$ and a periodic point $q$ of $f$, such that $W^u(y)\cap W^s(q)\neq\emptyset$.
\end{lem}

\begin{proof}
By Lemma~\ref{E contracted orbits for f}, there is a sequence of periodic points $\{q_n\}_{n\geq 1}$ which converges to $x$ such that $orb(q_n)$ accumulates to a subset of $\Lambda$, and
   \begin{displaymath}
     \prod_{i=0}^{\tau_n-1}\|Df|_{E(f^i(q_n))}\|< {\lambda_1}^{\tau},
   \end{displaymath}
where $\tau_n$ is the period of $q_n$, for any $n\geq 1$. Moreover, since $\Lambda$ is an aperiodic class, we have that $\tau_n$ goes to $+\infty$.

By the Pliss Lemma, there are $\lambda_1$-$E$-Pliss points on $orb(q_n)$. Consider all pairs of consecutive $\lambda_1$-$E$-Pliss points $(f^{k_i^n}(q_n),f^{l_i^n}(q_n))_{1\leq i\leq m_n}$ on $orb(q_n)$. We consider whether the sequence of numbers $(l_i^n-k_i^n)_{n\geq 1,1\leq i\leq m_n}$ is uniformly bounded or not.

\paragraph{Case 1.} If there is a number $N$ such that $0< l_i^n-k_i^n\leq N$, for all $i$ and all $n\geq 1$, then the set $\overline{\bigcup_{n\geq 1}orb(q_n)}$ is an $E$-contracted set. Hence any $q_n$ has a uniform stable manifold with dimension $dim(E)$ by Remark~\ref{uniform manifold for pliss points}. Since $x$ is a $\lambda_2$-bi-Pliss point, when $q_n$ is close enough to $x$, we have that $W^u(x)\cap W^s(q_n)\neq\emptyset$. Then we take $y=x$ and $q=q_n$.

\paragraph{Case 2.} We consider the case where the sequence of numbers $(l_i^n-k_i^n)_{n\geq 1,1\leq i\leq m_n}$ is not uniformly bounded. By considering a subsequence if necessary and to simplify the notations, we assume that the sequence of consecutive $\lambda_1$-$E$-Pliss points $(f^{k_1^n}(q_n),f^{l_1^n}(q_n))$ satisfies $l_1^n-k_1^n\rightarrow +\infty$ as $n\rightarrow +\infty$. By item $(1)$ of Lemma~\ref{property of pliss point}, any limit point $y\in\Lambda$ of the sequence $\{f^{l_1^n}(q_n)\}$ is a $\lambda_1$-bi-Pliss point. Hence when $f^{l_1^n}(q_n)$ is close enough to $y$, we have that $W^u(y)\cap W^s(f^{l_1^n}(q_n))\neq\emptyset$. Then we take $q=f^{l_1^n}(q_n)$.
\medskip

The proof of Lemma~\ref{contradiction} is now completed.
\end{proof}

\begin{proof}[End of the proof of Theorem A]
We prove Theorem A by contradiction. If neither $E$ is contracted nor $F$ is expanded, then by Lemma~\ref{contradiction}, we have that $W^u(y)\cap W^s(q)\neq\emptyset$ for some $y\in\Lambda$ and some periodic point $q$. Since $\Lambda$ is a Lyapunov stable aperiodic class, we have that $W^u(y)\subset \Lambda$, hence $q\in\Lambda$, which contradicts the fact that $\Lambda$ contains no periodic point. Hence one of the two bundles $E$ and $F$ is hyperbolic, which is a contradiction of the assumption.
\end{proof}

\section{Connecting a set and a point by periodic orbits: proof of Proposition~\ref{asymptotic connecting}}

In this section, we prove Proposition~\ref{asymptotic connecting}.

Consider a diffeomorphism $f$ that satisfies the properties stated in Theorem~\ref{generic properties}, a chain-transitive set $K$ of $f$, a point $x\in K$ satisfying $x\notin \alpha(x)$ (hence $x$ is not a periodic point) and a $C^1$-neighborhood $\mathcal{U}$ of $f$. By taking a smaller neighborhood if necessary, we assume that the elements of $\mathcal{U}$ are of the form $f\circ\phi$ with $\phi\in\mathcal{V}$, such that $\mathcal{V}$ is a $C^1$-neighborhood of $Id$ which satisfies the property (F):\\
\textit{(F) For any perturbations $\phi $ and $\phi'$ of $Id$ in $\mathcal{V}$ with disjoint support, the composed perturbation $\phi\circ \phi'$ is still in $\mathcal{V}$.}\\
For the $C^1$-neighborhood $\mathcal{U}$ of $f$, there is an integer $N$ given by Theorem~\ref{connecting lemma}.

We fix the triple $(f,\mathcal{U},N)$, and fix the three neighborhoods $U_x$ of $x$, $U$ of $\alpha(x)$ and $U_K$ of $K$ from now on. We consider two cases.
\begin{itemize}
\item \textbf{The non-periodic case:} there is a point $z\in\alpha(x)$ such that $z$ is not a periodic point of $f$ with period less than or equal to $N$.
\item \textbf{The periodic case:} any point contained in $\alpha(x)$ is a periodic point of $f$ with period less than or equal to $N$.
\end{itemize}

\subsection{The non-periodic case.}
We construct three perturbation neighborhoods at three points, and choose segments of orbits that connect them one by one and then we use the connecting lemma to get a periodic orbit by perturbations. We point out here that the perturbation neighborhoods are pairwise disjoint and the segments of orbits are also pairwise disjoint. Moreover, any perturbation neighborhood is disjoint with the segment of orbit that connects the two other perturbation neighborhoods.

\subsubsection{Choice of points and connecting orbits}

\paragraph{\textmd{\emph{The perturbation neighborhoods at $x$.}}} We take two small neighborhoods $V'_x\subset V_x\subset U_x$ of $x$, such that the following properties are satisfied:
\begin{itemize}
\item $(\bigcup_{i=0}^N f^i(\overline{V_x}))\cap \alpha(x)=\emptyset$,
\item $\bigcup_{i=0}^N f^i(V_x)\subset U_K$,
\item $V'_x\subset V_x$ satisfy Theorem~\ref{connecting lemma} for the triple $(f,\mathcal{U},N)$.
\end{itemize}

\paragraph{\textmd{\emph{The point $y$.}}} Since all periodic points of $f$ are hyperbolic, we take a small neighborhood $V\subset U$ of $\alpha(x)$, such that
\begin{itemize}
\item $V\subset U_K$ and $\overline{V}\cap (\bigcup_{i=0}^N f^i(\overline{V_x}))=\emptyset$,
\item there is no periodic point with period less than or equal to $N$ in $\overline{V}\setminus K$.
\end{itemize}

Since $K$ is a chain transitive set and $x\in K$, we have that $x\dashv_K \alpha(x)$. By item $3$ and $4$ of Theorem~\ref{generic properties}, we have that $x\prec_{U_K} \alpha(x)$. By Lemma~\ref{prec}, there is a point $y\in V\setminus \alpha(x)$, such that $x\prec_{U_K} y\prec_V \alpha(x)$, and $orb^+(y)\subset V$. Since $x\notin V$, we have that $y\notin \overline{orb^-(x)}$. Moreover, there is an integer $n_0$, such that, for any $i\geq n_0$, $f^{-i}(x)\subset V$.

\paragraph{\textmd{\emph{The perturbation neighborhoods at $y$, the connecting orbit from $x$ to $y$ and the number $L$.}}} By the choice the neighborhood $V$, the point $y$ is not a periodic point with period less than or equal to $N$. By the facts that $y\notin \overline{orb^-(x)}$ and $orb^+(y)\subset V$, there are two small neighborhoods $V'_y\subset V_y$ of $y$, such that the following properties are satisfied:
\begin{itemize}
\item $\bigcup_{i=0}^N f^i(V_y)\subset V$, which implies that $(\bigcup_{i=0}^N f^i(\overline{V_y}))\cap (\bigcup_{i=0}^N f^i(\overline{V_x}))=\emptyset$,
\item $\overline{orb^-(x)}\cap(\bigcup_{i=0}^N f^i(\overline{V_y}))=\emptyset$, which implies that $(\bigcup_{i=0}^N f^i(\overline{V_y}))\cap \alpha(x)=\emptyset$,
\item $V'_y\subset V_y$ satisfy Theorem~\ref{connecting lemma} for the triple $(f,\mathcal{U},N)$.
\end{itemize}
\medskip

Since $x\prec_{U_K} y$, there is a piece of orbit segment $\{w_1,f(w_1),\cdots,f^{n_1}(w_1)\}\subset U_K$, such that $w_1\in V'_x$ and $f^{n_1}(w_1)\in V'_y$. By the choice of $V_y$, one can see that $w_1\notin orb^-(x)$. Take $L=n_0+n_1+N$. Then we take an integer $m\in\mathbb{N}$.

\paragraph{\textmd{\emph{The point $z$ and the perturbation neighborhoods at $z$.}}} Now take $z\in\alpha(x)$ such that $z$ is not a periodic point of $f$ with period less than or equal to $N$. Then we have that $y\prec_V z$. By the fact that $z\in\alpha(x)$ and $(\bigcup_{i=0}^N f^i(\overline{V_y}))\cap \alpha(x)=\emptyset$, we can take two neighborhoods $V'_z\subset V_z$ of $z$, such that the following properties are satisfied:
\begin{itemize}
\item $\bigcup_{i=0}^N f^i(V_z)\subset V$, which implies that $(\bigcup_{i=0}^N f^i(\overline{V_z}))\cap (\bigcup_{i=0}^N f^i(\overline{V_x}))=\emptyset$,
\item $(\bigcup_{i=0}^N f^i(\overline{V_z}))\cap (\bigcup_{i=0}^N f^i(\overline{V_y}))=\emptyset$,
\item $\{w_1,f(w_1),\cdots,f^{n_1}(w_1)\}\cap (\bigcup_{i=0}^N f^i(\overline{V_z}))=\emptyset$,
\item $f^{-i}(x)\notin \bigcup_{i=0}^N f^i(V_z)$, for any $0\leq i\leq n_0+m$,
\item $V'_z\subset V_z$ satisfy Theorem~\ref{connecting lemma} for the triple $(f,\mathcal{U},N)$.
\end{itemize}

\paragraph{\textmd{\emph{The connecting orbits from $y$ to $z$ and from $z$ to $x$.}}} Since $y\prec_V z$, there is a piece of orbit segment $\{w_2,f(w_2),\cdots,f^{n_2}(w_2)\}\subset V$, such that $w_2\in V'_y$ and $f^{n_2}(w_2)\in V'_z$. By the choice of $V_y$, we have that $w_2\notin orb^-(x)$. By the choice of the neighborhood $V$, we have that $\{w_2,f(w_2),\cdots,f^{n_2}(w_2)\}\cap (\bigcup_{i=0}^N f^i(\overline{V_x}))=\emptyset$.

Since $z\in\alpha(x)$, there is $n_3$, such that $f^{-n_3}(x)\in V'_z$. Since $V_y\cap orb^-(x)=\emptyset$, we have that $\{f^{-n_3}(x),f^{-n_3+1}(x),\cdots,x\}\cap (\bigcup_{i=0}^N f^i(\overline{V_y}))=\emptyset$. Moreover, by the choice of $V_z$, we have that $n_3> n_0+m$.
\medskip

To sum up, we have obtained three pairwise disjoint perturbation neighborhoods $\bigcup_{i=0}^N f^i(V_x)$, $\bigcup_{i=0}^N f^i(V_y)$ and $\bigcup_{i=0}^N f^i(V_z)$, and three pairwise disjoint pieces of orbit segment $\{w_1,f(w_1),\cdots,\\
f^{n_1}(w_1)\}$, $\{w_2,f(w_2),\cdots,f^{n_2}(w_2)\}$ and $\{f^{-n_3}(x),f^{-n_3+1}(x),\cdots,x\}$ that connect the perturbation neighborhoods one by one. Moreover, all these perturbation neighborhoods and pieces of orbit segments are contained in the neighborhood $U_K$ of $K$, and each perturbation neighborhood is disjoint with the piece of orbit segment that connects the other two perturbation neighborhoods.

\subsubsection{The connecting process}
Now, by Theorem~\ref{connecting lemma}, we will do perturbations of $f$ in $\mathcal{U}$ on the three pairwise disjoint neighborhoods $\bigcup_{i=0}^N f^i(V_x)$, $\bigcup_{i=0}^N f^i(V_y)$ and $\bigcup_{i=0}^N f^i(V_z)$. By Remarque 4.3 of~\cite{bc} and the Property (F), one can get a diffeomorphism in $\mathcal{U}$ with the composition of the three perturbations.

\paragraph{\textmd{\emph{The perturbation at $x$.}}} The point $f^{-n_3}(x)$ has a positive iterate $x\in V'_x$ and the point $f^{n_1}(w_1)$ has a negative iterate $w_1\in V'_x$. By Theorem~\ref{connecting lemma}, there is a diffeomorphism $f_1\in\mathcal{U}$, such that:
\begin{itemize}
\item $f_1$ coincides with $f$ outside $\bigcup_{i=0}^N f^i(V_x)$,
\item the point $f^{n_1}(w_1)$ is on the positive orbit of $f^{-n_3}(x)$ under $f_1$.
\end{itemize}
Moreover, the piece of orbit segment $\{f^{-n_3}(x),f_1(f^{-n_3}(x)),\cdots f^{n_1}(w_1)\}$ under $f_1$ satisfies the following properties:
\begin{itemize}
\item it is contained in $\bigcup_{i=0}^{n_3}\{f^{-i}(x)\}\cup(\bigcup_{i=0}^{N-1}f^i(V_x))\cup\bigcup_{i=0}^{n_1}\{f^{i}(w_1)\}$,
\item it intersect $V_x$ and has at most $n_0+n_1$ points outside $V$, where $n_0+n_1<L$,
\item it contains the piece of orbit segment $\{f^{-n_0}(x),f^{-n_0-1}(x),\cdots,f^{-n_0-m}(x)\}$ under $f$.
\end{itemize}
We take a point $p\in V_x\cap \{f^{-n_3}(x),f_1(f^{-n_3}(x)),\cdots f^{n_1}(w_1)\}$, then the point $f^{n_1}(w_1)$ is on the positive orbit of $p$ and $f^{-n_3}(x)$ is on the negative of $p$ under $f_1$.

\paragraph{\textmd{\emph{The perturbation at $y$.}}} By the above construction, $f_1$ coincides with $f$ in $\bigcup_{i=0}^N f^i(V_y)$. Hence the piece of orbit $\{w_2,f(w_2),\cdots,f^{n_2}(w_2)\}$ is not modified. Under the iterate of $f_1$, the point $p$ has a positive iterate $f^{n_1}(w_1)\in V'_y$ and the point $f^{n_2}(w_2)$ has a negative iterate $w_2\in V'_y$. By Theorem~\ref{connecting lemma}, there is $f_2\in\mathcal{U}$, such that:
\begin{itemize}
\item $f_2$ coincides with $f_1$ outside $\bigcup_{i=0}^N f^i(V_y)$, hence $f_2$ coincides with $f$ outside $(\bigcup_{i=0}^N f^i(V_x))\cup(\bigcup_{i=0}^N f^i(V_y))$,
\item $f^{n_2}(w_2)$ is on the positive iterate of $p$ under $f_2$, and $f^{-n_3}(x)$ is on the negative of $p$ under $f_2$,
\item the piece of orbit segment $\{f^{-n_3}(x),f_2(f^{-n_3}(x)),\cdots,p,f_2(p),\cdots,f^{n_2}(w_2)\}$ under $f_2$ has at most $L=n_0+n_1+N$ points outside $V$,
\item the piece of orbit segment $\{f^{-n_0}(x),f^{-n_0-1}(x),\cdots,f^{-n_0-m}(x)\}$ under $f$ is contained in the piece of orbit segment $\{f^{-n_3}(x),f_2(f^{-n_3}(x)),\cdots,p,f_2(p),\cdots,f^{n_2}(w_2)\}$ under $f_2$.
\end{itemize}

\paragraph{\textmd{\emph{The perturbation at $z$.}}} By the above constructions, $f_2$ coincides with $f$ in $\bigcup_{i=0}^N f^i(V_z)$, and, under the iterate of $f_2$, the point $p$ has a positive iterate $f^{n_2}(w_2)\in V'_z$ and a negative iterate $f^{-n_3}(x)\in V'_z$. By Theorem~\ref{connecting lemma}, there is $g\in\mathcal{U}$, such that:
\begin{itemize}
\item $g$ coincides with $f_2$ outside $\bigcup_{i=0}^N f^i(V_z)$, hence $g$ coincides with $f$ outside $(\bigcup_{i=0}^N f^i(V_x))\cup(\bigcup_{i=0}^N f^i(V_y))\cup(\bigcup_{i=0}^N f^i(V_z))$,
\item the point $p\in U_x$ is a periodic point of $g$,
\item the piece of orbit segment $\{f^{-n_0}(x),f^{-n_0-1}(x),\cdots,f^{-n_0-m}(x)\}$ under $f$ is contained in $orb(p,g)$, hence $orb(p,g)$ has at least $m$ points contained in $V\subset U$,
\item $orb(p,g)$ has at most $L=n_0+n_1+N$ points outside $V\subset U$.
\end{itemize}
\medskip

This finishes the proof of Proposition~\ref{asymptotic connecting} in the non-periodic case.

\subsection{The periodic case.}
In this case, any point contained in $\alpha(x)$ is a periodic point with period less than or equal to $N$. By the assumption that all periodic point of $f$ is hyperbolic, we have that $\alpha(x)$ is a finite set. Since $\alpha(x)$ is chain transitive, this gives the following claim.

\begin{claim}\label{alpha x}
In this case, $\alpha(x)$ is a hyperbolic periodic orbit $orb(q)$, and $x\in W^u(orb(q))$.
\end{claim}

\subsubsection{Choice of points and connecting orbits}
By Claim~\ref{alpha x}, we have that $x\in W^u(orb(q))\setminus\{q\}$ for a hyperbolic periodic point $q$. To simplify the notations, we just assume that $q$ is a hyperbolic fixed point of $f$, but the general case is identical. Now $U$ is a neighborhood of $\{q\}=\alpha(x)$.

\paragraph{\textmd{\emph{The perturbation neighborhoods at $x$.}}} We take two small neighborhoods $V'_x\subset V_x\subset U_x$ of $x$, such that the following properties are satisfied:
\begin{itemize}
\item $q\notin\bigcup_{i=0}^N f^i(\overline{V_x})$,
\item $\bigcup_{i=0}^N f^i(V_x)\subset U_K$,
\item $V'_x\subset V_x$ satisfy Theorem~\ref{connecting lemma} for the triple $(f,\mathcal{U},N)$.
\end{itemize}

\paragraph{\textmd{\emph{The neighborhood $V$ and the point $y$.}}} By the hyperbolicity of the fixed point $q$, there is a neighborhood $V$ of $q$, such that, if the positive orbit of a point is contained in $\overline{V}$, then this point is in $W^s(q)$. Moreover, we can assume $V$ is small such that $\overline{V}\cap(\bigcup_{i=0}^N f^i(\overline{V_x}))=\emptyset$ and $\overline{V}\subset U$. By the assumption, we have that $x\dashv_K q$, hence by items $3$ and $4$ of Theorem~\ref{generic properties}, we have that $x\prec_{U_K} q$. By Lemma~\ref{prec}, there is a point $y\in V\setminus q$, such that $x\prec_{U_K} y\prec_V q$, and $orb^+(y)\subset V$. Then we can see that $y\in W^s(q)$. Moreover, since $x\notin V$, we have that $y\notin \overline{orb^-(x)}$.

\paragraph{\textmd{\emph{The perturbation neighborhoods at $y$, the connecting orbit from $x$ to $y$ and the number $L$.}}}
We take two neighborhoods $V'_y\subset V_y$ of $y$, such that the followings are satisfied:
\begin{itemize}
\item $\bigcup_{i=0}^N f^i(\overline{V_y})\subset V\setminus \overline{orb^-(x)}$, which implies that $(\bigcup_{i=0}^N f^i(\overline{V_y}))\cap (\bigcup_{i=0}^N f^i(\overline{V_x}))=\emptyset$,
\item $f^{j}(y)\notin \bigcup_{i=0}^N f^i(\overline{V_y})$, for any $j\geq N+1$,
\item $V'_y\subset V_y$ satisfy Theorem~\ref{connecting lemma} for the triple $(f,\mathcal{U},N)$.
\end{itemize}

Since $x\prec_{U_K} y$, there is a piece of orbit segment $\{w_1,f(w_1),\cdots,f^{n_1}(w_1)\}\subset U_K$, such that $w_1\in V'_x$ and $f^{n_1}(w_1)\in V'_y$.
Since $x\in W^u(p)$, there is an integer $n_0$, such that $f^{-i}(x)\in V$ for any $i\geq n_0$. By the fact that $orb^+(f^{N+1}(y))\cup orb^-(f^{-n_0}(x))\subset V$, and $\overline{V}\cap(\bigcup_{i=0}^N f^i(\overline{V_x}))=\emptyset$, we can see that the piece of orbit segment $\{w_1,f(w_1),\cdots,f^{n_1}(w_1)\}$ is disjoint with $orb^+(f^{N+1}(y))\cup orb^-(f^{-n_0}(x))$. Take $L=n_0+n_1+N$.

\subsubsection{The connecting process: to get a transverse homoclinic point}
Now we do the perturbations by Theorem~\ref{connecting lemma} and get a homoclinic point of $q$ contained in $U_x$.

\paragraph{\textmd{\emph{The perturbation at $x$.}}} The point $f^{-n_0}(x)$ has a positive iterate $x\in V'_x$ and the point $f^{n_1}(w_1)$ has a negative iterate $w_1\in V'_x$. By Theorem~\ref{connecting lemma}, there is a diffeomorphism $f_1\in\mathcal{U}$, such that:
\begin{itemize}
\item $f_1$ coincides with $f$ outside $\bigcup_{i=0}^N f^i(V_x)$,
\item $f^{n_1}(w_1)$ is on the positive orbit of $f^{-n_0}(x)$ under $f_1$,
\end{itemize}

Then there is a point $z\in V_x$, such that $f^{n_1}(w_1)$ is on the positive orbit of $z$ under $f_1$ and $f^{-n_0}(x)$ is on the negative orbit of $z$ under $f_1$. Assume that $f^{n_1}(w_1)=f_1^{m_1}(z)$ and $f^{-n_0}(x)=f_1^{-m_0}(z)$, we have that $m_1+m_0\leq n_1+n_0+N=L$. Moreover, the negative orbit of $f^{-n_0}(x)$ under $f$ is still the negative orbit of $f^{-n_0}(x)$ under $f_1$, hence $z\in W^u(q,f_1)$. Also, by the fact that $\bigcup_{i=0}^N f^i(\overline{V_y})\subset V\setminus \overline{orb^-(x)}$, one can see $(\bigcup_{i=0}^N f^i(\overline{V_y}))\cap \overline{orb^-(z,f_1)}=\emptyset$.

\paragraph{\textmd{\emph{The perturbation at $y$.}}} By the above construction, $f_1$ coincides with $f$ in $\bigcup_{i=0}^N f^i(V_y)$, and, under the iterate of $f_1$, the point $z$ has a positive iterate $f_1^{m_1}(p)=f^{n_1}(w_1)\in V'_y$ and the point $f^{N+1}(y)$ has a negative iterate $y\in V'_y$. By Theorem~\ref{connecting lemma}, there is $f_2\in\mathcal{U}$, such that:
\begin{itemize}
\item $f_2$ coincides with $f_1$ outside $\bigcup_{i=0}^N f^i(V_y)$, hence $f_2$ coincides with $f$ outside $(\bigcup_{i=0}^N f^i(V_x))\cup(\bigcup_{i=0}^N f^i(V_y))$,
\item the point $f^{N+1}(y)$ is on the positive iterate of $z$ under $f_2$.
\end{itemize}
\medskip

By the fact that $(\bigcup_{i=0}^N f^i(\overline{V_y}))\cap \overline{orb^+(f^{N+1}(y),f)}=\emptyset$, we can see that the positive orbit of $f^{N+1}(y)$ under $f$ is still  the positive orbit of $f^{N+1}(y)$ under $f_2$. Hence $z\in W^s(q,f_2)$. Since $(\bigcup_{i=0}^N f^i(\overline{V_y}))\cap \overline{orb^-(z,f_1)}=\emptyset$, we have that $orb^-(z,f_1)=orb^-(z,f_2)$, hence $z\in W^u(q,f_2)$. Then the point $z\in W^s(q,f_2)\cap W^u(q,f_2)$ is a homoclinic point of the hyperbolic fixed point $q$ of $f_2$.

By perturbing $f_2$ to a diffeomorphism $g\in\mathcal{U}$ with an arbitrarily $C^1$ small perturbation, we can assume that $z$ is a transverse homoclinic point of the hyperbolic fixed point $q$ of $g$. Moreover, the orbit of $z$ under $g$ is the same as that of $f_2$, hence the number of points of $orb(z,g)\setminus V$ is no more than $m_1+m_0$, hence no more than $L$.

\subsubsection{Periodic orbits shadowing the orbit of a homoclinic point: end of the proof} Now we have obtained a diffeomorphism $g\in\mathcal{U}$, a hyperbolic fixed point $q$ and a transverse homoclinic point $z\in W^s(q,g)\pitchfork W^u(q,g)$ whose orbit under $g$ has at most $L$ points outside $V$. Then the set $orb(z,g)\cup \{q\}$ is a hyperbolic set. By a standard argument with the $\lambda$-lemma and the Smale's homoclinic theorem, for any integer $m\in\mathbb{N}$, there is a periodic point $p\in V_x$ of $g$, such that $orb(p,g)$ has at most $L$ points outside $V$ and has at least $m$ points inside $V$. The proof of Proposition~\ref{asymptotic connecting} in the periodic case is finished by the fact that $V\subset U$.

\medskip

The proof of Proposition~\ref{asymptotic connecting} is completed.

\flushleft{\bf Xiaodong Wang} \\
School of Mathematical Sciences, Peking University, Beijing, 100871, P.R. China\\
Laboratoire de Math\'ematiques d'Orsay, Universit\'e Paris-Sud 11, Orsay 91405, France\\
\textit{E-mail:} \texttt{xdwang1987@gmail.com}\\

\end{document}